\documentclass{article}

\usepackage{amsmath, amsthm, amsfonts}
\usepackage{graphicx}
\usepackage[all]{xy}

\usepackage{amsthm}
\theoremstyle{plain}

\newtheorem{thm}{Theorem}
\newtheorem{lem}[thm]{Lemma}
\newtheorem{cor}[thm]{Corollary}

\theoremstyle{definition}
\newtheorem{exl}[thm]{Example}

\renewcommand{\phi}{\varphi}
\newcommand{\lift}{\widetilde}

\newcommand{\RT}{\text{\textit{RT}}}
\newcommand{\pd}{\partial}

\newcommand{\xof}[1]{\frac{\Delta}{\Delta #1}}
\newcommand{\fox}[1]{\frac{\pd}{\pd #1}}

\newcommand{\Reid}{\mathcal{R}}

\DeclareMathOperator{\Coin}{Coin}
\DeclareMathOperator{\ind}{ind}
\DeclareMathOperator{\sign}{sign}
\DeclareMathOperator{\id}{id}

\newcommand{\Z}{\mathbb{Z}}

\usepackage{hyperref}

\theoremstyle{plain}

\begin{document}
\bibliographystyle{hplain}

\title{A formula for the coincidence Reidemeister trace of selfmaps on
  bouquets of circles}
\author{{\sc P. Christopher Staecker}\thanks{\emph{Email:}
  cstaecker@messiah.edu} \\ 
Assistant Professor, Dept. of Mathematical Sciences, \\
Messiah College, Grantham PA, 17027.}

\maketitle

\begin{abstract}
We give a formula for the coincidence Reidemeister trace of selfmaps on
bouquets of circles in terms of the Fox calculus. Our formula reduces
the problem of computing the coincidence Reidemeister trace to the
problem of distinguishing doubly twisted conjugacy classes in free
groups.
\end{abstract}

\section{Introduction}
Fadell and Husseini, in \cite{fh83} proved the following:
\begin{thm}\label{fh}[Fadell, Husseini, 1983]
Let $X$ be a bouquet of circles, $G = \pi_1(X)$, and let $G_0$ be the
set of generators of $G$. 
If $f: X \to X$ induces the 
map $\phi:G \to G$, then there is some lift $\lift f:
\lift X \to \lift X$ to the universal covering space with 
\[ \RT(f,\lift f) = \rho\left(  1 - \sum_{a \in G_0} \frac{\pd}{\pd
  a}\phi(a)\right), \] 
where $\rho:\Z G \to \Z\Reid(f)$ is the linearization of the projection
into twisted conjugacy classes, and $\pd$ denotes the Fox derivative.
\end{thm}

Theorem \ref{fh} reduces the calculation of the Reidemeister trace
(and thus of the Nielsen number) in fixed point theory to the
computation of twisted conjugacy classes. Our goal for this paper is
to obtain a similar result in coincidence theory of selfmaps-- a
formula for the coincidence Reidemeister trace $\RT(f,\lift f, g,
\lift g)$ in terms of Fox derivatives which reduces the computation to
twisted conjugacy decisions.

The proof of Theorem \ref{fh} given in \cite{fh83} is brief, 
thanks to a natural trace-like formula for the
Reidemeister trace in fixed point theory. No such formula exists for
the coincidence Reidemeister trace, and this will complicate our
derivation considerably. Our argument is based on first specifying a
particular regular form for maps in Section 3 and for pairs of maps in
Section 4. In Section 5 we give our main result.

The author would like to thank Seungwon Kim,
Philip Heath, and Nirattaya Khamsemanan for helpful and encouraging
conversations, and Robert Brown for helpful comments on the paper.

\section{Preliminaries}
Throughout the paper, let $X$ be a bouquet of circles meeting at
the base point $x_0$. Let $G = \pi_1(X)$, a free group, and let $G_0$ be
the set of generators of $G$. Let $\lift X$ 
be the universal covering space of $X$ with projection $p_X:\lift X
\to X$, and choose once and for all a base point $\lift x_0 \in \lift
X$ with $p_X(\lift x_0) = x_0$.

Given maps $f,g:X \to X$ and their induced homomorphisms $\phi, \psi:G
\to G$, we 
define an equivalence relation on $G$ as follows: two elements
$\alpha, \beta \in G$ are \emph{[doubly] twisted conjugate} if 
\[ \alpha = \phi(\gamma)\beta \psi(\gamma)^{-1}. \]
The equivalence classes with respect to this relation are the
\emph{Reidemeister classes}, and we denote the set of such classes as
$\Reid(\phi,\psi)$. Let $\rho: G \to \Reid(\phi,\psi)$ be the
projection into Reidemeister classes.

For any pair of maps $f,g:X \to X$, denote their coincidence set by
$\Coin(f,g) = \{ x \in X \mid f(x)=g(x) \}$. The set of coincidence
points are partitioned into \emph{coincidence classes}
of the form $\Coin(\alpha^{-1}\lift f, \lift g)$, where $\alpha \in G$
and $\lift f,\lift g:\lift X \to \lift X$ are specified lifts of $f$
and $g$. Lemma 9 of \cite{stae08a} shows that $\Coin(\alpha^{-1}\lift
f, \lift g) = \Coin(\beta^{-1}\lift f, \lift g)$ if and only if
$\rho(\alpha) = \rho(\beta)$, and that $\Coin(\alpha^{-1}\lift f,
\lift g)$ and $\Coin(\beta^{-1}\lift f, \lift g)$ are disjoint if
$\rho(\alpha) \neq \rho(\beta)$. (The statement appears in a slightly
different form in \cite{stae08a} because a slightly different twisted
conjugacy relation is used. Trivial modifications will produce the
version used here.) 

Thus the coincidence classes are represented by Reidemeister classes
in $G$, and so each particular coincidence point has an associated
Reidemeister class. For $x \in \Coin(f,g)$, let $[x_{\lift f, \lift
g}] \in \Reid(\phi,\psi)$ denote the Reidemeister class
$\rho(\alpha)$ for which $x \in p_X(\Coin(\alpha^{-1}\lift f, \lift
g))$.

Let $f,g:X \to X$ be mappings with isolated coincidence points, and
for each coincidence point $x$ let $U_x \subset X$ be a neighborhood
of $x$ containing no other coincidence points. Then we define the
\emph{coincidence Reidemeister trace} as:
\[ \RT(f,\lift f, g, \lift g) = \sum_{x \in \Coin(f,g)}
\ind(f,g,U_x)[x_{\lift f, \lift g}], \]
where $\ind$ denotes the coincidence index. Indeed we can define a
local Reidemeister trace: for any open set $U$, define
\[ \RT(f, \lift f, g, \lift g, U) = \sum_{x \in \Coin(f,g,U)}
\ind(f,g,U_x)[x_{\lift f, \lift g}], \]
where $\Coin(f,g,U) = \Coin(f,g) \cap U$. Clearly this local
Reidemeister trace is equal to the nonlocal version if $U$ is taken
to be $X$, and has the following additivity property: if $V$ and $W$
are disjoint subsets of $U$ with $\Coin(f,g,U) \subset V \cup W$, then
\[ \RT(f,\lift f,g,\lift g, U) = \RT(f,\lift f, g, \lift g, V) +
\RT(f, \lift f, g, \lift g, W). \]

The coincidence index above is well known in the setting of maps from
one orientable manifold to another of the same dimension. In our setting, the
space $X$ is not a manifold, and so this index will in general be
undefined on sets containing the point $x_0$ (or sets whose images under
$f$ and $g$ contain $x_0$). Many of our
results (in particular Lemmas \ref{halfzone}, \ref{classes},
\ref{tophalflemma}, and \ref{bottomhalflemma}) are localized away from
$x_0$, and as such will apply directly to the case where $f, g:X
\to Y$ with $X$ and $Y$ being two different bouquets of circles.

Our restriction to the case of selfmaps allows us to sidestep the
complications at $x_0$ by reducing to the fixed point index, which is
defined for any ANR (and in particular is well known for bouquets of
circles). Obtaining a formula similar 
to our main result for non-selfmaps may require an extension of the
coincidence index to certain non-manifold settings, which is in
general a difficult problem. Gon\c{c}alves in \cite{gonc99}
defines an index for maps from a complex into a manifold of the same
dimension, but it is not integer-valued. No generalization of the
index to maps from one complex to another is known. 

Very little is known about the coincidence Reidemeister trace. The
term ``trace'' is in reference to the construction in fixed point
theory, in which the Reidemeister trace is given by the trace of a
certain matrix (see \cite{geog02}), but no trace-like formula is yet
available in coincidence theory. The above definition suffices to
define the Reidemeister trace only when the coincidence set is
isolated. The main result of \cite{stae08a} shows that a local coincidence
Reidemeister trace can be defined for general maps $f:X \to Y$ (with perhaps
nonisolated coincidence points) if $X$ and $Y$ are
orientable manifolds of the same dimension, and that this local
Reidemeister trace is uniquely characterized by five natural axioms,
among them additivity and homotopy invariance. 

Our setting (where $X$ and $Y$ are the same bouquet of circles) does
not precisely fit the setting of \cite{stae08a} because $X$ is
not a manifold. We will apply certain of the results, however, to
$\RT(f,\lift f, g, \lift g, U)$ in the
case where the neither $U$ nor $f(U)$ and $g(U)$ contain the 
point $x_0$, as in that case $U$ and its images can be
regarded as manifolds in their own right. In particular, for such a
set $U_x$, we will make use of the fact that
\[ \ind(f,g,U_x) = \sign \det(dg_x - df_x), \]
where as above $U_x$ is some subset of $U$ containing only one
coincidence point $x$, and $dg_x$ and $df_x$ denote the derivatives of
$f$ and $g$ (assuming that they exist).

As for defining the Reidemeister trace for any pair of maps $f,g:X \to
X$ (perhaps having nonisolated coincidence points) with lifts $\lift
f,\lift g:\lift X \to \lift X$, we first change $f$ and $g$ by a
homotopy to $f', g'$ so that they have isolated coincidence points
(that this is possible will be a consequence of our 
construction in Theorem \ref{formula}). Now the homotopies of $f,g$ to
$f',g'$ can be lifted to a homotopy of $\lift f, \lift g$ to some
lifts $\lift f', \lift g'$. We then define
\[ \RT(f,\lift f, g, \lift g) = \RT(f', \lift f', g', \lift g'). \]
That this is well defined will be a consequence of the homotopy invariance
of the coincidence index and the homotopy-relatedness of coincidence
classes (see Lemma 14 of \cite{stae08a}).

We will now review the neccessary properties of the Fox calculus (see
e.g.\cite{cf63}). If $\{x_i\}$ are the generators of a free group $G$,
then the operators $\frac{\pd}{\pd x_i}: G \to \Z G$ are defined by:
\begin{align*}
\frac{\pd}{\pd x_i} 1 &= 0, \\
\frac{\pd}{\pd x_i} x_j &= \delta_{ij}, \\
\frac{\pd}{\pd x_i}(uv) &= \frac{\pd}{\pd x_i}u + u\frac{\pd}{\pd
  x_i}v,
\end{align*}
where $\delta_{ij}$ is the Kronecker delta, and $u,v \in G$ are any
words. Two important formulas can be obtained from the
above:
\begin{align*}
\frac{\pd}{\pd x_i}x_i^{-1} &= -x_i^{-1}, \\
\frac{\pd}{\pd x_i}(h_1 \dots h_n) &= \sum_{k=1}^n h_1 \dots h_{k-1}
\frac{\pd}{\pd x_i} h_k.
\end{align*}

\section{A regular form for mappings}
In this section we will describe a standard form for selfmaps of
$X$. Each circle of $X$ is represented by some generator of 
the fundamental group. For each generator $a \in G_0$, let $|a|
\subset X$ be the circle represented by $a$ (including the point
$x_0$).

For simplicity in our notation, we parameterize each circle by the
interval $[0,1]$ with endpoints identified. The circles of $X$ will be 
parameterized so that the base point $x_0$ is identified
with 0 (or equivalently, with 1). For any generator $a \in
G$ and any $x \in [0,1]$, let $[x]_a$ denote the point of $|a|
\subset X$ which has coordinate $x$. Interval-like subsets of $X$ will
be denoted e.g. $(x_1,x_2)_a$ for the subset of points in $|a|$ 
parameterized by the interval $(x_1, x_2)$. 

Homotopy classes of mappings of $X$ are
characterized by their induced mappings on the fundamental groups.
Consider the example where $G = \langle a,b \rangle$ and the mapping
$f:X \to X$ induces $\phi:G \to G$ with 
\[ \phi(a) =  ab^{-1}a^{-1}b^2. \]

Geometrically speaking, the above formula for $\phi$ indicates that $f$
is homotopic to a map which maps some interval $(0, x_1)_a$ bijectively
onto $|a|-x_0$, maps some interval $(x_1, x_2)_a$ bijectively onto
$|b|-x_0$ (in the ``reverse direction''), and so on.

We can represent the action of this map on $|a|$ pictorially as in
Figure \ref{firstpeg}, 
\begin{figure}
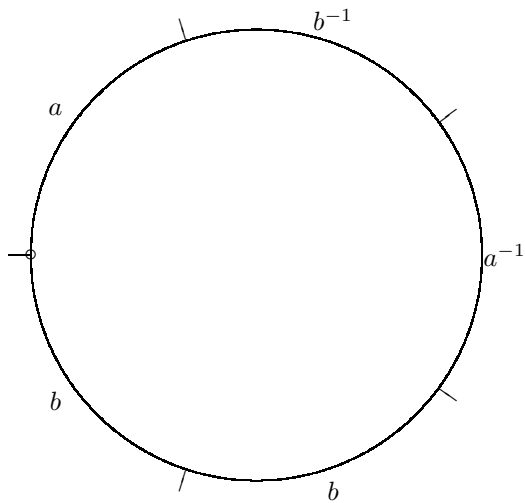

\[
\xy
*\xycircle(30,30){};
(-30, 0)*{\circ};
(-26.69, 19.39)*{a};
(-9.27,28.53)*{}; (-10.19,31.38)*{} **\dir{-};
(10.19, 31.38)*{b^{-1}};
(24.27,17.63)*{}; (26.69,19.39)*{} **\dir{-};
(33, 0)*{a^{-1}};
(24.27,-17.63)*{}; (26.69,-19.39)*{} **\dir{-};
(10.19, -31.38)*{b};
(-9.27,-28.53)*{}; (-10.19,-31.38)*{} **\dir{-};
(-26.69, -19.39)*{b};
(-30,0)*{}; (-33,0)*{} **\dir{-};
\endxy
\]
\caption{Diagram of the action on $|a|$ of a map with
$\phi(a) = ab^{-1}a^{-1}b^2$. The base point $x_0$ is circled at left.}
\label{firstpeg}
\end{figure}
where in this case $x_i = [i/5]_a$, and the label on each interval
indicates that the interval is being mapped bijectively onto the
corresponding circle. Note that sliding the points $x_i$ around the
circle will not change the homotopy class of $f$, provided that $x_0$
does not move, no $x_i$ ever moves across another, and the ordering of
the labels is preserved. 

Thus any map  $f:X \to X$ is homotopic to a map which is
characterized as follows: for each generator $a \in G_0$, specify
$n_a$ intervals $I^a_1, \dots I^a_{n_a}$ together with labels $\{h^a_1,
\dots, h^a_{n_a}\}$, where each of $h^a_i$ are letters of $G$ (a
\emph{letter} of $G$ is an element which is either a generator or the
inverse of a generator of $G$). These intervals will be 
called \emph{intervals of $f$}, and the labels will be called
the \emph{labels of $f$}.

Since we are concerned only with homotopy classes of maps,
we may assume that $f$ maps $(x_i, x_{i+1})_a$ affine
linearly onto $(0,1)_b$ for some generator $b\in G$.
A map specified by intervals and labels which is linear on each 
interval in this way will be called \emph{regular}.

Specifying a map $f$ by intervals and labels gives precise information
which can be used to compute the derivatives of $f$ at any point. For
any interval $I = (x_i, x_{i+1})_a$, define $w(I)$, the \emph{width}
of $I$, as the real number $x_{i+1}-x_i$.
\begin{lem}\label{derivatives}
Let $f:X \to X$ be a regular map, let $I$ be an interval of $f$
labeled by $b \in G$, and let $x$ be any point of $I$. Then we have:
\begin{itemize}
\item If $b$ is a generator of $G$, then 
\[ df_x = \frac1{w(I)} \]
\item If $b$ is the inverse of a generator of $G$, then
\[ df_x = -\frac1{w(I)}, \]
\end{itemize}
where we have abused our notation slightly, writing the derivatives
(typically represented as $1\times 1$ matrices) as real numbers.
\end{lem}
\begin{proof}
Let $I = (x_i, x_{i+1})_a$. In the case where $b$ is a generator of
$G$, we can compute the 
restriction $f:I \to (0, 1)_b$ as
\[ f(x) = \frac{1}{w(I)} (x -x_i), \]
and the derivative is as desired. In the case where $b$ is the inverse
of a generator, $f$ will reverse orientation, and the restriction will
be
\[ f(x) = \frac{1}{w(I)}(x_{i+1}-x) \]
again giving the desired derivative.
\end{proof}

\section{Regular pairs of mappings}
Let $G_0$ be the set of generators of $G$, and let 
let $\Phi = \{u_a \mid a \in G_0\}$, $\Psi = \{v_a \mid a
\in G_0 \}$ be ordered sets of unreduced words in $G$ of length at least
2. We will construct from these sets two maps $f$ and $g$ in a
standardized way.

We construct $g$ from $\Psi$ as follows, each $v_a$ determining $g$'s
behavior on the circle $|a|$: Let $m_a$ be the length of $v_a$. Let $g$'s first
interval on $|a|$ be $(0, 1/2)_a$, and let the remaining $m_a-1$
intervals be equally spaced over $(1/2, 1)_a$. These $m_a$
intervals are to be labeled respectively by the $m_a$ letters of $v_a$.

We will construct $f$ from $\Phi$ similarly, this time using a single wide
interval on $(1/2,1)_a$ and several evenly spaced intervals on $(0,
1/2)_a$. Our construction of $f$ is slightly complicated by the fact
that we will also require $f$ to be constant in
small neighborhoods of $x_0$ and $[1/2]_a$. On each circle $|a|$, define $f$
as follows: let $n_a$ be the length of $u_a$. 
Fixing some small $\epsilon > 0$, our first
interval in $|a|$ will be $(0, \epsilon)_a$, followed by $n_a-1$
evenly spaced intervals on over $(\epsilon, 1/2-\epsilon)_a$,
followed by the intervals $(1/2 - \epsilon, 1/2 + \epsilon)_a$, $(1/2 +
\epsilon, 1-\epsilon)_a$, and $(1-\epsilon, 1)_a$. These intervals will be 
labeled, respectively, by $1$, the letters of $u_a$ except the
last, 1, the last letter of $u_a$, and 1. (Labeling an interval by 1
indicates that $f$ is constant on that interval.)

\begin{figure}
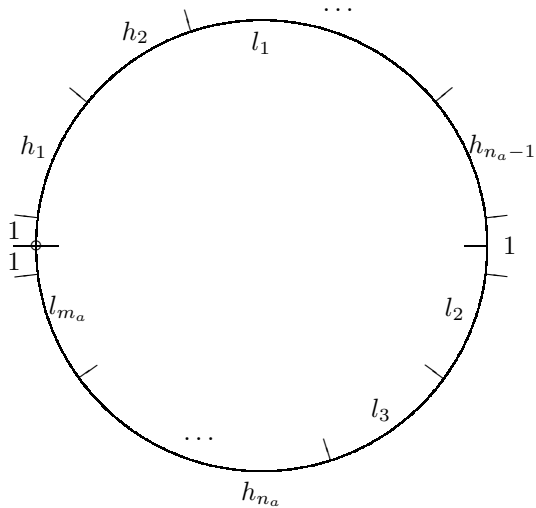

\[
\xy
*\xycircle(30,30){};
(-30, 0)*{\circ};
(-30,0)*{}; (-33,0)*{} **\dir{-};
(-32.93, 2.07)*{1};
(-29.76,3.75)*{}; (-32.73,4.13)*{} **\dir{-};
(-30.28, 13.1)*{h_1};
(-23.11,19.12)*{}; (-25.42,21.03)*{} **\dir{-};
(-16.79, 28.4)*{h_2};
(-9.27,28.53)*{}; (-10.19,31.38)*{} **\dir{-};
(10.19, 31.38)*{\dots};
(23.11,19.12)*{}; (25.42,21.03)*{} **\dir{-};
(30.28, 13.1)*{\quad h_{n_a-1}};
(29.76,3.75)*{}; (32.73,4.13)*{} **\dir{-};
(33, 0)*{1};
(29.76,-3.75)*{}; (32.73,-4.13)*{} **\dir{-};
(0, -33)*{h_{n_a}};
(-29.76,-3.75)*{}; (-32.73,-4.13)*{} **\dir{-};
(-32.93, -2.07)*{1};
(-30,0)*{}; (-27,0)*{} **\dir{-};
(0, 27)*{l_1};
(30,0)*{}; (27,0)*{} **\dir{-};
(25.67, -8.34)*{l_2};
(24.27,-17.63)*{}; (21.84,-15.87)*{} **\dir{-};
(15.87, -21.84)*{l_3};
(9.27,-28.53)*{}; (8.34,-25.67)*{} **\dir{-};
(-8.34, -25.67)*{\dots};
(-24.27,-17.63)*{}; (-21.84,-15.87)*{} **\dir{-};
(-25.67, -8.34)*{l_{m_a}};
\endxy
\]
\caption{Typical action of a regular pair on a single circle $|a|$,
  where $u_a = h_1 \dots h_n$ and $v_a = l_1 \dots l_m$. Intervals and
  labels for $f$ are marked on the outside, and intervals and labels
  for $g$ are marked on the inside.\label{standardpair}}
\end{figure}

A diagram of typical mappings constructed above is given in Figure
\ref{standardpair}. We say that the $f$ and $g$ constructed above are
the \emph{regular 
pair given by $\Phi$ and $\Psi$}. Typically we think of our sets
$\Phi$ and $\Psi$ as coming from a pair of homomorphisms $\phi, \psi:G
\to G$, letting 
$\Phi = \{ \phi(a) \mid a \in G_0 \}$ and $\Psi = \{ \psi(a) \mid a
\in G_0 \}$. In such a case we will say that $f$ and $g$ are the
regular pair given by $\phi$ and $\psi$. By our construction it is
clear that any pair of maps are homotopic to the regular
pair given by their induced homomorphisms on the fundamental group.

For example, if $\phi, \psi: G \to G$ are homomorphisms with
$G = \langle a, b, c \rangle$ such that 
\[ \phi(a) = aba^{-1}bc, \quad \psi(a) = c^2ab^{-1}a, \]
then the regular pair $f,g$ given by $\phi$ and $\psi$ is pictured in
Figure \ref{nicepair}.

\begin{figure}
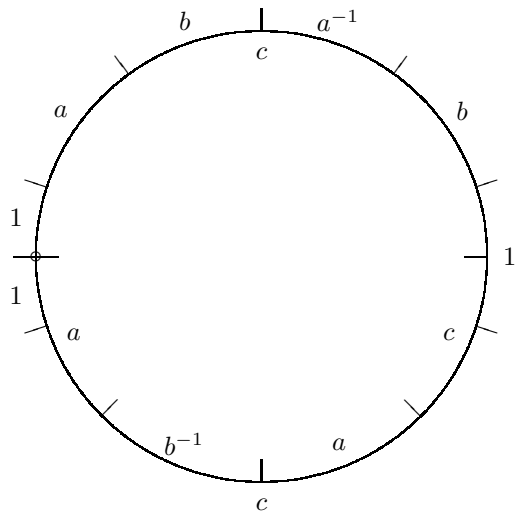

\[
\xy
*\xycircle(30,30){};
(-30, 0)*{\circ};
(-30,0)*{}; (-33,0)*{} **\dir{-};
(-32.59, 5.16)*{1};
(-28.53,9.27)*{}; (-31.38,10.19)*{} **\dir{-};
(-26.69, 19.39)*{a};
(-17.63,24.27)*{}; (-19.39,26.69)*{} **\dir{-};
(-10.19, 31.38)*{b};
(0,30)*{}; (0,33)*{} **\dir{-};
(10.19, 31.38)*{a^{-1}};
(17.63,24.27)*{}; (19.39,26.69)*{} **\dir{-};
(26.69, 19.39)*{b};
(28.53,9.27)*{}; (31.38,10.19)*{} **\dir{-};
(33, 0)*{1};
(28.53,-9.27)*{}; (31.38,-10.19)*{} **\dir{-};
(0, -33)*{c};
(-28.53,-9.27)*{}; (-31.38,-10.19)*{} **\dir{-};
(-32.59, -5.16)*{1};
(0, 27)*{c};
(30,0)*{}; (27,0)*{} **\dir{-};
(24.94, -10.33)*{c};
(21.21,-21.21)*{}; (19.09,-19.09)*{} **\dir{-};
(10.33, -24.94)*{a};
(0,-30)*{}; (0,-27)*{} **\dir{-};
(-10.33, -24.94)*{b^{-1}};
(-21.21,-21.21)*{}; (-19.09,-19.09)*{} **\dir{-};
(-24.94, -10.33)*{a};
(-30,0)*{}; (-27,0)*{} **\dir{-};
\endxy
\]
\caption{Diagram for the action on $|a|$ of the regular pair given by
  homomorphisms with $\phi(a)=aba^{-1}bc$ and $\psi(a)=c^2ab^{-1}a$. The
  first map is indicated by intervals and labels on the outside of the
  circle, and the second is indicated on the inside. \label{nicepair}}
\end{figure}

The slight generality obtained by allowing $\Phi$ and $\Psi$ to
include nonreduced words is an important one which we use below.

\section{The Reidemeister trace of a regular pair}
Let $f,g:X \to X$ be a regular pair of maps, and choose
lifts $\lift f$ and $\lift g$ of $f$ and $g$ respectively so that
$\lift f(\lift x_0) = \lift g(\lift x_0)$. In this section we will describe
a method for calculating $\RT(f, \lift f, g, \lift g)$.

Since $f$ and $g$ are linear on intervals which never have the
same endpoints (except at $x_0$), we know that their coincidence set
consists of isolated points. We consider first the points $x_0$ and
$[1/2]_a$, which are coincidence points by construction. 

\begin{lem}\label{zerozone}
Let $f,g:X \to X$ be the regular pair given by $\Phi = \{u_a\}$ and
$\Psi = \{v_a\}$. Let $I_0$ be the union of the point $x_0$ and all intervals of
$f$ having $x_0$ as an endpoint. If $v_a = aw_a a$ for a word $w_a \in
G$, then 
\[ \ind(f,g,I_0) = 1. \]
\end{lem}
\begin{proof}
Since $v_a$ begins and ends with $a$ and $g(x_0) = x_0$, the map $g$ will map some an
initial segment of $|a|$ to itself, and also will map some terminal
segment of $|a|$ to itself. Thus there is some $U \subset I_0$ on
which $g$ is homotopic to the identity map. 

Thus we have
\[ \ind(f,g,I_0) = \ind(f,\id, U), \]
and this is simply the fixed point index of $f$ at $x_0$. The fixed
point index at the ``wedge point'' of a bouquet of circles is well
studied (see e.g. \cite{wagn99}), and is equal to 1.
\end{proof}

\begin{lem}\label{halfzone}
Let $f, g:X \to X$ be the regular pair given by $\Phi = \{u_a\}$ and $\Psi
= \{v_a\}$. Let $I_a$ be the interval of $f$ containing $[1/2]_a$. If
$v_a = b_a b_a^{-1} w_a$ for some generator $b_a \in G$ and a word
$w_a \in G$, then 
\[ \ind(f,g,I_a) = 0 \]
for any generator $a \in G$.
\end{lem}
\begin{proof}
Since the intervals of $g$ on either side of $[1/2]_a$ are labeled
by inverse elements, our map $g$ can be changed by a homotopy in a
neighborhood of $[1/2]_a$ to be constant, with constant value some $x
\in |b_a|-x_0$. But $f$ will have constant value $x_0$ on $I_a$, and so
will be coincidence free with the resulting map homotopic to
$g$. Thus the index on $I_a$ is zero as desired.
\end{proof}

Now we turn to the coincidence points other than $x_0$ and $[1/2]_a$.
\begin{lem}\label{classes}
If $f, g$ is the regular pair given by $\Phi = \{ u_a \}$ and $\Psi =
\{ v_a \}$ 
and $x\in |a|$ is a coincidence point other than
$x_0$ or $[1/2]_a$ lying in an interval labeled $h_i$ by $f$ and $l_j$ by $g$, then
\[ [x_{\lift f, \lift g}] = \begin{cases} 
\rho(h_1 \dots h_{i-1} (l_1 \dots l_{j-1})^{-1}) &\text{if $h_i = l_j$,}
\\
\rho(h_1 \dots h_{i-1}(l_1
  \dots l_j)^{-1}) &\text{if $h_i = l_j^{-1}$,} \end{cases}
\]
where 
\[ u_a = h_1 \dots h_{n_a}, \quad v_a = l_1 \dots l_{m_a}. \]
\end{lem}
\begin{proof}
Suppose that $x$ lies in the intervals $I = (x_i, x_{i+1})_a$ and $J =
(z_j, z_{j+1})_a$ of $f$ and $g$ respectively.
Since $x$ is a coincidence point, we know that $f(I) = g(J)$, which means
that either $h_i = l_j$ or $h_i = l_j^{-1}$, and thus our cases in the
statement of the Lemma are exhaustive.

Define a path by a positively oriented arc from $x_0$ to
$x_i$, and lift this path to $\lift X$, starting at the initial point
$\lift x_0$. Define $\lift x_i$ as the terminal point of this path,
and we have 
\[ \lift f(\lift x_i) = \alpha\lift y_0, \]
where $\alpha$ is the covering transformation corresponding to
the word $h_1 \dots h_{i-1} \in G$, and $\lift y_0 = \lift f(\lift
x_0) = \lift g(\lift x_0)$. Similarly defining $\lift z_j$, we have 
\[ \lift g(\lift z_j) = \beta \lift y_0, \]
where $\beta$ is the covering transformation corresponding to the word
$l_1 \dots l_{j-1} \in G$. Thus we have
\[ \lift g(\lift z_j) = \beta \alpha^{-1} \lift f(\lift x_i). \]

In the case where $h_i = l_j$, we have $f(x_i) = g(z_j)$ and
$f(x_{i+1}) = g(z_{j+1})$. Thus any group element $\sigma \in
G$ with $x \in p_X (\Coin(\sigma^{-1} \lift f, \lift g))$ must have
$\sigma^{-1} \lift f(\lift x_i) = \lift g(\lift z_j)$. Then by the
above formula we have $\sigma^{-1} = \beta \alpha^{-1}$, and thus that
$[x_{\lift f, \lift g}] = \rho(h_1 \dots h_{i-1}(l_1 \dots l_{j-1})^{-1})$
as desired. 

Now we turn to the case where $h_i = l_j^{-1}$, in which
$f(x_i) = g(z_{j+1})$ and $f(x_{i+1}) = g(z_j)$. Similar to $\lift z_j$ above, we can define $\lift z_{j+1}$ as the
terminal point of a lifted path from $x_0$ to $z_{j+1}$, and we have
\[ \lift g(\lift z_{j+1}) = \gamma \alpha^{-1}
\lift f(\lift x_i), \]
where $\gamma$ is the covering transformation corresponding to $l_1
\dots l_j \in G$. Thus any group
element $\sigma$ with $x \in p_X(\Coin(\sigma^{-1}\lift f, \lift g))$ must
have $\sigma^{-1} \lift f(\lift x_i) = \lift g(\lift z_{j+1})$, and
we have $\sigma^{-1} = \gamma \alpha^{-1}$, and
thus that $[x_{\lift f, \lift g}] = \rho(h_1 \dots h_{i-1}(l_1 \dots
l_j)^{-1})$ as desired. 
\end{proof}

\begin{lem}\label{tophalflemma}
Let $f$ and $g$ be the regular pair given by $\Phi = \{u_a\}$ and
$\Psi = \{v_a\}$, with $v_a$ having a generator as its first letter,
and let $I$ be an interval of $f$ with $I \subset (0,1/2)_a$ 
for some circle $|a| \subset X$. Then
\[ \RT(f, \lift f, g, \lift g, I) = -\rho( h_1 \dots
h_{i-1} \frac{\pd}{\pd l_1} h_i), \]
where $u_a = h_1 \dots h_n$, the interval $I$ is labeled by $h_i$, and
$l_1$ is the first letter of $v_a$.
\end{lem}
\begin{proof}
Since $I$ contains at most one coincidence point (both $f$ and $g$ are
affine linear on $I$), we know that $\RT(f,\lift f, g, \lift g, I)$ is
zero if there is no coincidence point in $I$, and otherwise
\[ \RT(f, \lift f, g, \lift g, I) = \ind(f,g,I) [x_{\lift f, \lift
	g}], \]
where $x$ is the unique coincidence point in $I$.

We prove the lemma in three cases, treating the various possible
relationships between $h_i$ and $l_1$. Either: $h_i$ is not $l_1$ or
its inverse, or they are equal, or they are inverses.

In the first case, $h_i \neq l_1^{\pm 1}$, and so the maps $f$ and $g$
have no coincidence point in $I$. Thus 
\[ \RT(f,\lift f, g, \lift g, I) = 0 = -\rho(h_1 \dots h_{i-1} \frac{\pd}{\pd  l_1}h_i). \]

In the second case, $h_i = l_1$, and so $\frac{\pd}{\pd l_1} h_i =
1$. Thus we must show that 
\[ \RT(f,\lift f, g, \lift g, I) = -\rho(h_1 \dots h_{i-1}). \]
Since $h_i = l_1$, the maps $f$ and $g$ have 
one coincidence point $x$ in $I$. By our calculation in Lemma
\ref{derivatives}, we know that $df_x > 2$, since the
width of the interval of $f$ containing $x$ is less than 1/2. Since
$g$'s interval containing $x$ is of width exactly $1/2$, we know that
$dg_x = 2$, and thus  
\[ \ind(f,g,I) = \sign(\det(dg_x - df_x)) = -1. \]
By Lemma \ref{classes} we have that $[x_{\lift f, \lift g}] = \rho(h_1
\dots h_{i-1})$. Thus we have 
\[ \RT(f,\lift f, g, \lift g, I) = \ind(f,g,I) [x_{\lift
	f, \lift g}] = -\rho(h_1 \dots h_{i-1}) \]
as desired.

In the third case we assume $h_i = l_1^{-1}$. In this case, we have
$\frac{\pd}{\pd l_1}h_i = -h_i$, and so we must show that 
\[ \RT(f, \lift f, g, \lift g, I) = \rho(h_1 \dots
h_i). \]
As above, we have one coincidence
point $x$, but this time the derivatives are: $df_x < -2$ and 
$dg_x = 2$ as before. Thus $\ind(f,g,I) = 1$. By Lemma \ref{classes}
we have $[x_{\lift f, \lift g}] = \rho(h_1 \dots h_i)$, and thus that
\[ \RT(f, \lift f, g, \lift g, I) = \ind(f, g,
x)[x_{\lift f, \lift g}] = \rho(h_1 \dots h_i)\] 
as desired.
\end{proof}

\begin{lem} \label{bottomhalflemma}
With notation as in Lemma \ref{tophalflemma},
let $I$ be some interval of $g$ with $J \subset (1/2, 1)_a$ for some
circle $|a|\subset X$, and assume that $h_n$ is a generator of $G$. Then
\[ \RT(f, \lift f, g, \lift g, J) = \rho\left( h_1 \dots h_{n-1} (l_1
\dots l_{i-1} \frac{\pd}{\pd h_n}l_i)^{-1}\right), \]
where $v_a = l_1 \dots l_m$, and the interval $J$ is labeled by $l_i$.
\end{lem}
\begin{proof}
The proof is very similar to that of Lemma \ref{tophalflemma}. Again
we split the argument into three cases: either $l_i$ is not $h_n$ or
its inverse, or $l_i = h_n$, or $l_i = h_n^{-1}$.

In the first case, $l_i \neq h_n^{\pm 1}$ and so $f$ and $g$ have no
coincidences in $J$. Thus we have
\[ \RT(f, \lift f, g, \lift g, J) = 0 = \rho\left( h_1 \dots h_{n-1} (l_1
\dots l_{i-1} \frac{\pd}{\pd h_n}l_i)^{-1}\right), \]
as desired.

In the second case, $l_i = h_n$, and so $\frac{\pd}{\pd h_n} l_i =
1$, and we must show
\[ \RT(f, \lift f, g, \lift g, J) = \rho(h_1 \dots h_{n-1}(l_1 \dots
l_{i-1})^{-1} ). \]
In this case $f$ and $g$ have a single coincidence point $x\in
J$. As in our arguments for Lemma \ref{tophalflemma} we
can compute that $df_x = 2+\epsilon$ for some arbitrarily small
$\epsilon$, and $dg_x > 2$, and thus that $\ind(f,g,J)
= 1$. By Lemma \ref{classes} we have
$[x_{\lift f, \lift g}] = \rho(h_1 \dots h_{n-1}(l_1 \dots l_{i-1})^{-1})$. Thus
\[ \RT(f, \lift f, g, \lift g, J) =
\ind(f,g,J)[x_{\lift f, \lift g}] = \rho(h_1 \dots h_{n-1}(l_1 \dots
l_{i-1})^{-1}) \]
as desired.

For the third case we have $l_i = h_n^{-1}$, and so $\frac{\pd}{\pd
  h_n} l_i = -l_i$, and we must show 
\[ \RT(f, \lift f, g, \lift g, J) = - \rho(h_1 \dots h_{n-1}(l_1
  \dots l_i)^{-1} ). \]
As usual we compute derivatives and find that
$df_x = 2+\epsilon$ and $dg_x < -2$ and thus that $\ind(f,g,J) = -1$.
Lemma \ref{classes} gives $[x_{\lift f, \lift g}] = \rho(h_1 \dots
h_{n-1}(l_1 \dots l_i)^{-1})$ as desired.
\end{proof}

Let $i:\Z G \to \Z G$ be the involution defined by
\[ i\left(\sum_k c_kg_k\right) = \sum_k c_k g_k^{-1}. \]

\begin{thm}\label{formula}
Let $f, g: X \to X$ be maps which induce the homomorphisms $\phi, \psi:
G \to G$. Then there are lifts $\lift f$ and $\lift g$ such
that for any particular generator $b \in G$, we have
\[ \RT(f, \lift f, g, \lift g) = \rho\left( 1 - \sum_{a \in G_0} 
\frac\pd{\pd a}\phi(a) + \phi(a)\psi(a)^{-1} - \phi(a)a^{-1}
i(\frac{\pd}{\pd a} \psi(a)) \right). \]
\end{thm}
\begin{proof}
For each circle $|a| \subset X$, define unreduced words $u_{a} =
\phi(a)a^{-1}a$ and $v_{a} = aa^{-1}\psi(a)a^{-1}a$. Now let
$f', g': X \to X$ be the regular pair given by $\Phi = \{ u_a \}$
and $\Psi = \{ v_a\}$ which is homotopic to the
original pair $f,g$. By construction this regular pair satisfies the
hypotheses of Lemmas \ref{zerozone}, \ref{halfzone},
\ref{tophalflemma}, and \ref{bottomhalflemma}.
Let $\lift f$ and $\lift g$ be the lifts of $f$
and $g$ which fix $\lift x_0$. Now lifting the homotopies of $f$ and $g$ to
$f'$ and $g'$ gives lifts $\lift f'$ and $\lift g'$ with
\[ \RT(f,\lift f, g, \lift g) = \RT(f', \lift f', g', \lift g'). \]
Write $\phi(a) =
h^a_1 \dots h^a_{n_a}$ and $\psi(a) = l^a_1 \dots l^a_{m_a}$.
A diagram of the action of $f'$ and $g'$ on a typical circle $|a|$ is
given in Figure \ref{proofpegs}. 

\begin{figure}
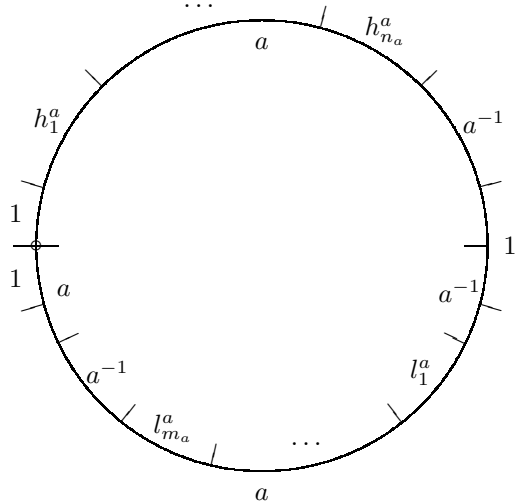

\[
\xy
*\xycircle(30,30){};
(-30, 0)*{\circ};
(-30,0)*{}; (-33,0)*{} **\dir{-};
(-32.71, 4.34)*{1};
(-28.96,7.82)*{}; (-31.85,8.6)*{} **\dir{-};
(-28.5, 16.61)*{h^a_1};
(-21.21,21.21)*{}; (-23.33,23.33)*{} **\dir{-};
(-8.4, 31.91)*{\dots};
(7.82,28.96)*{}; (8.6,31.85)*{} **\dir{-};
(16.44, 28.61)*{h^a_{n_a}};
(21.21,21.21)*{}; (23.33,23.33)*{} **\dir{-};
(28.61, 16.44)*{\,\,\,a^{-1}};
(28.96,7.82)*{}; (31.85,8.6)*{} **\dir{-};
(33, 0)*{1};
(28.96,-7.82)*{}; (31.85,-8.6)*{} **\dir{-};
(0, -33)*{a};
(-28.96,-7.82)*{}; (-31.85,-8.6)*{} **\dir{-};
(-32.71, -4.34)*{1};
(-30,0)*{}; (-27,0)*{} **\dir{-};
(0, 27)*{a};
(30,0)*{}; (27,0)*{} **\dir{-};
(26.31, -6.05)*{a^{-1}};
(27.06,-12.94)*{}; (24.35,-11.64)*{} **\dir{-};
(21.12, -16.81)*{l^a_1};
(18.68,-23.47)*{}; (16.81,-21.12)*{} **\dir{-};
(5.88, -26.34)*{\dots};
(-6.72,-29.23)*{}; (-6.05,-26.31)*{} **\dir{-};
(-11.8, -24.28)*{l^a_{m_a}};
(-18.68,-23.47)*{}; (-16.81,-21.12)*{} **\dir{-};
(-21.12, -16.81)*{\,\,a^{-1}};
(-27.06,-12.94)*{}; (-24.35,-11.64)*{} **\dir{-};
(-26.34, -5.88)*{a};
\endxy
\]
\caption{Diagram of the action of $f'$ and $g'$ on $|a|$ in the proof of Theorem
  \ref{formula}. \label{proofpegs}}
\end{figure}

In order to compute the Reidemeister trace, we partition $X$ into
several intervals: let $I^a_1, \dots, I^a_{n_a}$ be the intervals of $f'$
in $(0,1/2)_a$ labeled by $h^a_1, \dots, h^a_{n_a}$, and let $J^a_1, \dots,
J^a_{m_a}$ be the intervals of $g'$ in $(1/2, 1)_a$ labeled by $l^a_1,
\dots, l^a_{m_a}$. Let $I_0$ and $I_a$ be as in Lemmas
\ref{zerozone} and \ref{halfzone}. 

We define four more intervals to cover the remaining
portions of $X$: let $K^a_1$ be the interval of $f$ in $(0,1/2)_a$
which follows $I^a_{n_a}$ and is labeled by $a^{-1}$. Let $V^a$ be
the first interval of $g$ in $(1/2,1)_a$ (this interval is labeled by
$a^{-1}$) and we set $K^a_2$ to be the interior of $V^a - I_a$. Let
$K^a_3$ be the interval of $g$ following $J^a_{m_a}$, this interval is
labeled $a^{-1}$ by $g$. Finally let $U^a$ be the last interval of $g$, and
let $K^a_4$ be the interior of $U^a - I_0$. Since our regular maps are
constructed to have no coincidences at interval endpoints except for
$x_0$ and the $[1/2]_a$, we have
\[ \Coin(f',g') \subset I_0 \cup \bigcup_{a \in G_0}
\left(I_a \cup \bigcup_{i=1}^{n_a}I^a_i \cup 
\bigcup_{j=1}^{m_a} J^a_j \cup \bigcup_{k=1}^4 K^a_k
\right), \]
and so we can compute $\RT(f',\lift f', g', \lift g')$ as a sum of the
Reidmeister traces over the various intervals in the above union.

By Lemma \ref{halfzone} the index (and thus the Reidemeister trace)
will be zero on each $I_a$. Thus our summation simplifies to 
\begin{equation}\label{bigsum}
\RT(f',\lift f',g',\lift g') = \RT(I_0) + \sum_{a \in G_0} \left(
\sum_{i=1}^{n_a} \RT(I^a_i) + \sum_{j=1}^{m_a} \RT(J^a_j) + \sum_{k=1}^4
\RT(K^a_k) \right), 
\end{equation}
where for brevity we write $\RT(\cdot)$ for $\RT(f',\lift f',g',\lift
g',\cdot)$. 

The set $I_0$ contains a single coincidence point $x_0$. We know that
 $x_0 = p_X(\lift x_0)$, with $\lift x_0 \in \Coin(\lift f', \lift g')$.
Thus $[x_{\lift f', \lift g'}] = \rho(1)$, and by Lemma \ref{zerozone} the
index on $I_0$ is 1. Thus we have
\begin{equation}\label{I_0sum}
\RT(I_0) = \ind(f', g', I_0) [x_{\lift f', \lift g'}] = \rho(1). 
\end{equation}

Now by Lemma \ref{tophalflemma} we can compute the Reidemeister traces
on the $I^a_i$:
\[ \RT(I^a_i) = -\rho\left(h^a_1 \dots h^a_{i-1}\frac{\pd}{\pd a}h^a_i \right), \]
and so 
\begin{equation}\label{Isum}
\sum_{i=1}^{n_a} \RT(I^a_i) =
-\rho\left(\frac{\pd}{\pd a} (h^a_1 \dots h^a_{n_a})\right) =
-\rho\left( \frac{\pd}{\pd a}\phi(a)\right) .
\end{equation}

We can also apply Lemma \ref{tophalflemma} to compute $\RT(K^a_1)$,
which is labeled by $f$ with $a^{-1}$ and by $g$ with $a$. We obtain
\[ \RT(K^a_1) = - \rho(\phi(a)\frac{\pd}{\pd a}a^{-1}) =
\rho(\phi(a)a^{-1}). \]

Similarly we can use Lemma \ref{bottomhalflemma} to compute the
Reidemeister traces on the remaining intervals. For the $J^a_j$, we have
\[ \RT(J^a_j) = \rho\left( \phi(a)a^{-1} (aa^{-1}l^a_1 \dots
l^a_{j-1}\frac{\pd}{\pd a} l^a_j)^{-1} \right), \] 
and thus
\begin{equation}\label{Jsum}
\sum_{j=1}^{m_a} \RT(J^a_j) = \rho(\phi(a)a^{-1} i(\frac{\pd}{\pd
  a}\psi(a))). 
\end{equation}

Both $f$ and $g$ map $K^a_2$ to some initial segment of $|a|$, and so
there will be one coincidence point in $K^a_2$, and Lemma
\ref{bottomhalflemma} gives
\[ \RT(K^a_2) = \rho( \phi(a)a^{-1}(a\frac{\pd}{\pd a}a^{-1})^{-1}) =
-\rho(\phi(a)a^{-1}). \]

On $K_3^a$, we have $g(K_3^a) = |a|$ and $f(K_3^a) \subset |a|$, so
there will be one coincidence point, and again Lemma
\ref{bottomhalflemma} gives
\begin{align*}
\RT(K^a_3) &= \rho(\phi(a)a^{-1}(aa^{-1}\psi(a) \frac{\pd}{\pd
a}a^{-1})^{-1}) = -\rho(\phi(a)a^{-1}(\psi(a)a^{-1})^{-1}) \\
&= -\rho(\phi(a)\psi(a)^{-1}). 
\end{align*}

On $K_4^a$, we can change $g$ by a homotopy to make it coincidence
free with $f$, and thus 
\[ \RT(K^a_4) = 0. \]

Now summing (\ref{I_0sum}), (\ref{Isum}) and (\ref{Jsum}) with our
formulas for the $K_k^a$, our sum (\ref{bigsum}) gives
\begin{align*} 
\RT(f', \lift f', g', \lift g') &= \rho\left( 1 + \sum_{a \in G_0}
-\frac{\pd}{\pd a} \phi(a) + \phi(a)a^{-1}i(\frac{\pd}{\pd
 a}\psi(a)) - \phi(a)\psi(a)^{-1}\right) \\
&= \rho\left( 1 - \sum_{a \in G_0} 
\frac\pd{\pd a}\phi(a) + \phi(a)\psi(a)^{-1} - \phi(a)a^{-1}
i(\frac{\pd}{\pd a} \psi(a)) \right).
\end{align*}
\end{proof}

We end this section with an interpretation of the above formula with
respect to a variation on the Fox calculus. For generators $\{x_i\}$
of $G$, define the operators $\xof{x_i}: G \to \Z G$ as
follows:
\begin{align*}
\xof{x_i} 1 &= 0 \\
\xof{x_i} x_j &= \delta_{ij} \\
\xof{x_i} (uv) &= \left(\xof{x_i} u\right) v + \xof{x_i} v
\end{align*}

Analogous to the properties given for the Fox calculus we can derive:
\begin{lem}
If $\Delta$ denotes the operator above and $h_i$ are letters of $G$, we have:
\begin{align*}
\xof{x_i}x_i^{-1} &= -x_i, \\
\xof{x_i}(h_n \dots h_1) &= \sum_{k=1}^n \left( \xof{x_i} h_k \right)
h_{k-1} \dots h_1
\end{align*}
\end{lem}
\begin{proof}
The first statement follows by the computation:
\[ 0 = \xof{x_i}(x_i x_i^{-1}) = \left( \xof{x_i}x_i \right) x_i^{-1}
+ \xof{x_i}x_i^{-1} = x_i^{-1} + \xof{x_i}x_i^{-1}. \]

The second statement is proved by induction on $n$. For $n=1$ the
statement is clear. For the inductive case, we have
\begin{align*}
\xof{x_i}(h_n \dots h_1) &= \xof{x_i}(h_n \dots h_2h_1) = \left(
\xof{x_i}(h_n \dots h_2)\right) h_1 + \xof{x_i}h_1 \\
&= \sum_{k=2}^n \left( \xof{x_i} h_k \right) h_{k-1}\dots h_2 h_1 +
\xof{x_i}h_1 \\
&= \sum_{k=1}^n \left( \xof{x_i} h_k \right) h_{k-1} \dots h_1
\end{align*}
\end{proof}

We will use one further property, relating our new operator to the
ordinary Fox calculus operator:
\begin{lem}\label{foxxof}
For any word $w \in G$, we have
\begin{equation}\label{foxxofeq}
\xof{x_j}w = x_j^{-1} i\left( \fox{x_j} w \right) w. 
\end{equation}
\end{lem}
\begin{proof}
Our proof is by induction on the length of $w$. If $w$ is the trivial
element, then we have $\xof{x_j}w = 0$ by definition. Note also that 
\[ x_j^{-1}i\left( \fox{x_j} 1 \right) 1 = 0, \]
as desired.

For the inductive case, write $w = u v$. Then the left hand side of
(\ref{foxxofeq}) is
\[ \xof{x_j}(uv) = \left(\xof{x_j} u\right)v + \xof{x_j}v = x_j^{-1}i\left(
\fox{x_j} u\right) uv + x_j^{-1}i\left(\fox{x_j} v\right)v. \]
while the right hand side of (\ref{foxxofeq}) becomes
\begin{align*} 
x_j^{-1} i\left( \fox{x_j} (u v)\right) uv 
&= x_j^{-1}i\left( \fox{x_j}u + u\fox{x_j}v\right) uv \\
&= x_j^{-1}i\left( \fox{x_j}u\right) uv + x_j^{-1} i\left(\fox{x_j}v\right)u^{-1}uv
\end{align*}
as desired.
\end{proof}

Theorem \ref{formula} takes a nice form when expressed in terms of the
above operators: 
\begin{cor}
With notation as in Theorem \ref{formula}, we have
\[ \RT(f, \lift f, g, \lift g) = \rho\left( 1 - \sum_{a \in G_0}
\frac{\pd}{\pd a} \phi(a) - \xof{a}\psi(a) + \phi(a)\psi(a)^{-1} \right). \]
\end{cor}
\begin{proof}
It suffices to show that 
\[ \rho\left( \phi(a)a^{-1}i(\fox a \psi(a)) \right) = \rho\left( \xof
a \psi(a) \right). \]
Noting that $\phi(x)y = \phi(x)(y \psi(x))\psi(x)^{-1}$ for any $x,y
\in G$ gives the well known identity $\rho(\phi(x)y) =
\rho(y\psi(x))$. This identity, together with Lemma \ref{foxxof}, gives
\[ \rho\left(\phi(a)a^{-1}i(\fox a \psi(a))\right) = \rho\left(a^{-1}i(\fox a
\psi(a)) \psi(a)\right) = \rho\left(\xof a \psi(a)\right). \]
\end{proof}

\section{Some examples}
First we will show how Theorem \ref{fh} is a simple consequence of our
main result. Letting $\psi$ be the identity map gives
$\rho(\phi(a)a^{-1}) = \rho(a^{-1}\psi(a)) = \rho(1)$, and Theorem
\ref{formula} gives
\begin{align*} 
\RT(f,\lift f) &= \rho\left( 1 - \sum_{a \in G_0}
\frac{\pd}{\pd a} \phi(a) - \xof{a} a + \phi(a)a^{-1} 
 \right) \\ 
&= \rho\left( 1 - \sum_{a \in G_0} \frac{\pd}{\pd a}\phi(a) - 1 
+ 1 \right) \\ 
&= \rho\left( 1 - \sum_{a \in G_0} \frac{\pd}{\pd a}\phi(a) \right), 
\end{align*}
as desired.

Our formula also gives the classical formula for the coincidence
Nielsen number on the circle:
\begin{exl}
Let $G = \langle a \rangle$, and let $f$ and $g$ be maps which induce
the homomorphisms
\[ \phi(a) = a^n, \quad \psi(a) = a^m. \]
Without loss of generality, we will assume that $n\ge m$. Our formula
then gives
\begin{align*} 
\RT(f,\lift f,g,\lift g) &= \rho( 1 - \frac{\pd}{\pd a} a^n + \xof a a^m
- a^n a^{-m}) \\
&= \rho(1 - (1 + \dots + a^{n-1}) + (1 + \dots + a^{m-1}) - a^{n-m}) \\
&= \rho(1 - a^m - \dots - a^{n-1} - a^{n-m})
\end{align*}
A simple calculation shows that $\rho(a^i) = \rho(a^j)$ if and only if
$i\equiv j \mod n-m$. Thus $\rho(1) = \rho(a^{n-m})$, and all other
terms in the above sum are in distinct Reidemeister classes. 
Thus the Nielsen number is $n-m$, as desired.
\end{exl}

We conclude with one nontrivial computation of a coincidence Reidemeister
trace of two selfmaps of the bouquet of three circles. 
\begin{exl}
Let $X$ be a space with fundamental group $G = \langle a,b,c \rangle$,
and let $f$ and $g$ be maps which induce homomorphisms as follows:
\[ \phi: \begin{array}{rcl} a & \mapsto & acb^{-1} \\
b & \mapsto & ab \\
c & \mapsto & b \end{array} \quad 
\psi: \begin{array}{rcl} a & \mapsto & a^{-1}cb^{-1} \\
b & \mapsto & c \\
c & \mapsto & b^{-1}a\end{array} \]

Our formula gives
\begin{align*} 
\RT(f,\lift f, g, \lift g) &= \rho(1 - (1 + a^{-1}cb^{-1} + a^2) - (a
+ abc^{-1}) - (ba^{-1}b)) \\
&= \rho(-a - a^2 - abc^{-1} - ba^{-1}b - a^{-1}cb^{-1}),
\end{align*}
and we must decide the twisted conjugacy of the above elements. We use
the technique from \cite{stae08b} of abelian and nilpotent quotients.

Checking in the abelianization suffices to show that $a$
is not twisted conjugate to any of the other terms. We also see
that $a^2$ and $ba^{-1}b$ are twisted conjugate in the abelianization,
and our computation reveals that $\rho(a^2)=\rho(ba^{-1}b)$ with
conjugating element $\gamma = ac^{-1}$. Similarly, we find that
$\rho(a^2) = \rho(abc^{-1})$ by the element $\gamma = ab^{-1}$.

It remains to decide whether or not $a^2$ and $a^{-1}cb^{-1}$ are
twisted conjugate, and a check in the class 2 nilpotent quotient shows
that they are not. Thus
\[ \RT(f,\lift f, g, \lift g) = -\rho(a) - 3\rho(a^2) -
\rho(a^{-1}cb^{-1}) \]
is a fully reduced expression for the Reidemeister trace, and so in
particular the Neilsen number is 3.
\end{exl}


\begin{thebibliography}{1}

\bibitem{cf63}
R.~Crowell and R.~Fox.
\newblock {\em Introduction to Knot Theory}.
\newblock Springer, 1963.

\bibitem{fh83}
E.~Fadell and S.~Husseini.
\newblock The {N}ielsen number on surfaces.
\newblock {\em Contemporary Mathematics}, 21:59--98, 1983.

\bibitem{geog02}
R.~Geoghegan.
\newblock {N}ielsen fixed point theory.
\newblock In R.~J. Daverman and R.~B. Sher, editors, {\em The Handbook of
  Geometric Topology}, pages 499--521. North-Holland, 2002.

\bibitem{gonc99}
D.~L. Gon\c{c}alves.
\newblock Coincidence theory for maps from a complex into a manifold.
\newblock {\em Topology and Its Applications}, 92:63--77, 1999.

\bibitem{stae08a}
P.~C. Staecker.
\newblock Axioms for a local {R}eidemeister trace in fixed point and
  coincidence theory on differentiable manifolds.
\newblock 2007, arxiv eprint 0704.1891.

\bibitem{stae08b}
P.~C. Staecker.
\newblock Computing twisted conjugacy classes in free groups using nilpotent
  quotients.
\newblock 2007, arxiv eprint 0709.4407.

\bibitem{wagn99}
J.~Wagner.
\newblock An algorithm for calculating the {N}ielsen number on surfaces with
  boundary.
\newblock {\em Transactions of the American Mathematical Society}, 351:41--62,
  1999.

\end{thebibliography}
\end{document}